\documentclass[11pt]{amsart}
\usepackage{amsfonts, amsthm, amsmath, amssymb}
\usepackage[all]{xy}
\usepackage[top=1in,bottom=1in,left=1in,right=1in]{geometry}
\usepackage{tikz}
\usepackage[pdftex]{hyperref}

\newcommand{\bC}{\mathbb{C}}

\newcommand{\bM}{\mathbb{M}}

\newcommand{\bQ}{\mathbb{Q}}
\newcommand{\bR}{\mathbb{R}}

\newcommand{\bZ}{\mathbb{Z}}

\newcommand{\cV}{\mathcal{V}}

\newcommand{\fg}{\mathfrak{g}}
\newcommand{\fgl}{\mathfrak{gl}}

\newcommand{\fH}{\mathfrak{H}}
\newcommand{\fm}{\mathfrak{m}}

\newcommand{\fsl}{\mathfrak{sl}}

\newcommand{\fu}{\mathfrak{u}}

\newcommand{\simto}{\overset{\sim}{\to}}

\newcommand{\ad}{\operatorname{ad}}
\newcommand{\Aut}{\operatorname{Aut}}

\newcommand{\Tr}{\operatorname{Tr}}

\newtheorem{thm}{Theorem}[section]
\newtheorem*{conj}{Conjecture}
\newtheorem{prop}[thm]{Proposition}
\newtheorem{lem}[thm]{Lemma}
\newtheorem{cor}[thm]{Corollary}

\theoremstyle{definition}
\newtheorem{defn}[thm]{Definition}

\theoremstyle{remark}
\newtheorem{rem}[thm]{Remark}

\numberwithin{equation}{section}

\begin{document}

\title{Fricke Lie algebras and the genus zero property in Moonshine}
\author{Scott Carnahan}

\begin{abstract}
We give a new, simpler proof that the canonical actions of finite groups on Fricke-type Monstrous Lie algebras yield genus zero functions in Generalized Monstrous Moonshine, using a Borcherds-Kac-Moody Lie algebra decomposition due to Jurisich.  We describe a compatibility condition, arising from the no-ghost theorem in bosonic string theory, that yields the genus zero property.  We give evidence for and against the conjecture that such a compatibility for symmetries of the Monster Lie algebra gives a characterization of the Monster group.
\end{abstract}

\maketitle

\tableofcontents

\section{Introduction}

Modular functions and forms that satisfy a genus zero property play a central role in both old and new Moonshines.  However, we do not have a satisfactory global explanation for their appearance.  In general, the exceptional objects that appear in Moonshine appear to naturally come from conformal field theory, but it appears that working with physical aspects of the genus zero property requires us to leave that relatively safe world.  Mathematically, this seems to require specialized functors out of some nice subcategories of vertex operator algebras.  To date, the targets of such functors seem to involve either string theory or quantum gravity in 3-dimensional asymptotically Anti-de Sitter spacetime.

In this paper, we consider the bosonic string side, and in particular, a quantization functor that yields Borcherds-Kac-Moody Lie algebras equipped with a compatibility groupoid identifying root spaces.  This groupoid structure is sufficiently restrictive that it yields the genus zero property for the functions that appear in Norton's Generalized Moonshine conjecture.  In contrast to the first proof of Norton's conjecture \cite{GM4}, we describe general properties of Lie algebras equipped with compatibility groupoids, and we give a new, simpler description of the Lie algebras, following a decomposition method due to Jurisich.

Genus zero modular curves were connected to the Monster sporadic simple group $\bM$ even before Moonshine was really started: In \cite{O74}, Ogg observed that the quotient $X_0(p)^+$ of the modular curve $X_0(p)$ by its Fricke involution $\tau \mapsto \frac{-1}{p\tau}$ is genus zero if and only if the prime $p$ divides the order of $\bM$, and famously offered a bottle of Jack Daniels to whoever could give a good explanation.  It was only a few years later that McKay found a connection between the representation theory of $\bM$ and the modular $J$-function, which was enhanced in \cite{CN79} to a conjectured correspondence between elements of $\bM$ and genus zero modular functions.

Specifically, Conway and Norton conjectured the existence of a natural graded representation $V = \bigoplus_{n=0}^\infty V_n$ of the Monster $\bM$, such that for any $g \in \bM$, the graded character $\sum_{n=0}^\infty \Tr(g|V_n) q^{n-1}$ is the $q$-expansion of a specific holomorphic function $T_g(\tau)$ on the complex upper half-plane $\fH$.  They observed that all of the proposed expansions satisfied two properties:

\begin{enumerate}
\item The functions are Hauptmoduln, i.e., there is a discrete subgroup $\Gamma_g < SL_2(\bR)$ such that $T_g$ generates the function field of the corresponding quotient $\Gamma_g \backslash \fH$ of the upper half-plane.  In particular, the quotient Riemann surface is complex-analytically isomorphic to a sphere with finitely many punctures.
\item The coefficients satisfy some complicated recursion relations on their coefficients, which they named ``replication identities''.
\end{enumerate}

The initial formulation of replication in Section 8 of \cite{CN79} used equivariant Hecke operators applied to power series with coefficients in a representation ring, but without using those words.  It seems that this formulation was largely forgotten in favor of concrete identities on characters for a long time.  Hecke operators in Moonshine were recalled in a few later papers, e.g., \cite{F96} and in \cite{G09}, but they did not directly play a major role until they were employed to produce genus zero functions in \cite{GM1}.

Norton \cite{N84} reformulated the replicability condition in terms of Grunsky coefficients, which had been used in unimodular function theory for close to 50 years:  One starts with a formal power series $f(q) = q^{-1} + \sum_{n > 0} a_n q^n$ and defines the Grunsky matrix $\{ H_{m,n}\}_{m,n \in \bZ_{>0}}$ by the bivarial transform:
\[ \log \frac{f(p)-f(q)}{p^{-1}-q^{-1}} = - \sum_{m,n=1}^\infty H_{m,n} p^m q^n.\]
The power series $f$ is then replicable if and only if the value of $H_{m,n}$ depends only on $\gcd(m,n)$ and $mn$.  Norton noted that replicable functions are completely determined by their first 25 coefficients, and a proof appears in \cite{CN95}.  Furthermore, he claimed that the Moonshine functions satisfied a stronger notion that he called complete replicability (equivalent to the condition that all replicates are replicable).  Completely replicable functions are even more highly determined than replicable functions, as they are fixed by the first seven coefficients of all of the replicates.  Norton suggested that the bivarial transform is evidence of the existence of a doubly-graded algebraic structure that controls Monstrous Moonshine.


Borcherds's solution of the Monstrous Moonshine conjecture \cite{B92} uses replicability in the following way: Starting with the Monster module $V^\natural$ (constructed in \cite{FLM88}), one obtains a Lie algebra $\fm$ with $\bM$-action by employing a bosonic string quantization functor, and by applying a twisted denominator identity to $\fm$, one obtains complete replicability of trace functions, also known as McKay-Thompson series.  Unpublished work of Koike showed that Conway and Norton's proposed functions are all completely replicable.  From this, the conjecture follows from checking that the seven homogeneous subspaces of $V^\natural$ of lowest degree admit the predicted $\bM$-module structure.

By examining Borcherds's proof, we see that the bivarial transform is more or less the logarithm of the Weyl-Kac-Borcherds denominator identity, so $\fm$ is indeed the algebraic structure that fits into Norton's observation.  Furthermore, as we suggested earlier, replicability of a power series is equivalent to the condition that a modified Hecke operator sends the series to a monic degree $n$ polynomial in itself.  We find in \cite{GM1} that such Hecke operators appear naturally in twisted denominator identities.

We come to the main point of this paper, which is a derivation of the Hecke-monic property from group actions on Lie algebras by more transparent means.  Both Borcherds's proof of the Monstrous Moonshine conjecture, and the Lie-theoretic results in \cite{GM1} use the homology of the positive half of a Lie algebra, and the computation of this homology requires substantial structure theory of Borcherds-Kac-Moody Lie algebras.  However, Borcherds's computation can be avoided by appealing to a decomposition of $\fm$ due to Jurisich in \cite{J98}.  We apply the same idea to decompose the Lie algebras that appear in Generalized Monstrous Moonshine, and this substantially simplifies the derivation of twisted denominator identities in Section \ref{sec:decomposition} and the Hecke-monic property in Section \ref{sec:hecke}.

Returning to Ogg's Jack Daniels problem, we see that Borcherds's theorem solves half of it: The proof explains why $X_0(p)^+$ has genus zero for any prime $p$ dividing $|\bM|$.  Indeed, for any prime $p$ dividing the order of $\bM$ there is an element of order $p$ whose McKay-Thompson series is a Hauptmodul for some quotient of $X_0(p)$ or $X_0(p^2)$ (the latter case only possible if $p=3$), and the genus of $X_0(p)^+$ is then forced to be zero.   In the other direction, we would like to know why the Monster is so big that its order is a multiple of all fifteen of the relevant primes, or equivalently, why $V^\natural$ has so many symmetries, and this question seems to be wide open.  A closely related question, suggested to me by Terry Gannon, is ``why do small holomorphic vertex operator algebras admit the most exotic symmetries in mathematics?''  If we order the holomorphic vertex operator algebras by increasing central charge, the first nontrivial group we encounter is the Lie group $E_8$ at central charge 8, and the first nontrivial finite group we encounter (assuming the uniqueness of $V^\natural$) is the Monster at central charge 24.  While this phenomenon is fascinating, and a good answer to the question may help resolve Ogg's problem, we do not seem to have enough data now to draw strong conclusions.

Later work has yielded more clarity about the appearance of the genus zero property in the theory of replicable functions.  First, on the level of modularity, Yves Martin showed \cite{M94} that any completely replicable function that is ``$J$-final'', i.e., has a replicate that is the $J$-function, is invariant under some $\Gamma_0(N)$.  Kozlov \cite{K94} showed that completely replicable functions with suitable periodicity in the replicates satisfy modular equations of many orders.  Cummins and Gannon \cite{CG97} built on this result to show that completely replicable functions of finite order are either highly degenerate Laurent polynomials in $q$ (sometimes called ``modular fictions'') or Hauptmoduln.  This result was extended to cover all replicable functions of finite order in \cite{GM1}, essentially by the same techniques, but with some extra bookkeeping.

We now have a better understanding of what allows the distinguished $\bM$ action on the Monster Lie algebra $\fm$ to yield the genus zero condition, in part because we have many new examples of similar behavior, coming from Generalized Monstrous Moonshine.  As we will explain in Section \ref{sec:no-ghost}, the main idea is that the action is ``quantization compatible'', i.e., that string quantization identifies the action on many different root spaces.  We then define notions of q-compatible symmetry and q-compatible action, even for Lie algebras that don't necessarily come from conformal field theory, and show that q-compatible actions yield genus zero functions.

Using q-compatibility, we can hope to shed some light on the vague, still-open half of Ogg's question, namely why the Monster has elements of prime order $p$ for all $p$ for which $X_0(p)^+$ is genus zero.  Let $G$ be the group of q-compatible automorphisms of $\fm$.  We know that the Monster is a subgroup of $G$, by virtue of Borcherds's construction of $\fm$, but Proposition \ref{prop:prime-order-automorphisms} asserts that $G$ contains no elements whose order is a prime that doesn't divide the order of the Monster, essentially by the genus zero property.  We can therefore break up this part of Ogg's question into a puzzle and a conjecture.
\vspace{2ex}

\noindent\textbf{Puzzle:} Is there an intrinsic way to show, without resorting to the existence of $V^\natural$, that for any prime $p$ such that $X_0(p)^+$ is genus zero, there exists a q-compatible automorphism of $\fm$ of order $p$, or even better, an automorphism for which we obtain the Hauptmodul for $X_0(p)^+$?

\begin{conj}
$G \cong \bM$.
\end{conj}

If this conjecture is true, then we have a new characterization of the Monster in terms of a machine that directly generates genus zero functions.  We have the following evidence supporting the conjecture:

\begin{enumerate}
\item As we mentioned, $\bM \subseteq G$, by virtue of the quantization functor applied to $V^\natural$.
\item Any element of $G$ yields a completely replicable series (with complex coefficients), and in the finite-order case, a Hauptmodul.  This places an absolute bound on the order of finite-order elements of $G$, and as we mentioned, severely restricts the primes that can be orders of elements of $G$.  In particular, $G$ does not contain a copy of $U(1)$.
\item We point out in Lemma \ref{lem:diagram-automorphism-group} that the group of all diagram automorphisms of $\fm$ is infinite dimensional, of the form $\prod_{n=-1}^\infty GL_{c(n)}(\bC)$,
where the ranks $c(n)$ are the coefficients of $J$.  However, we show in Proposition \ref{prop:injectivity-for-q-compatible-maps} that the projection of the subgroup $G$ into $\prod_{n=1}^5 GL_{c(n)}(\bC)$ is injective, using a categorically enhanced version of complete replicability.
\end{enumerate}

For evidence against the conjecture, perhaps the strongest is the fact that there are many non-monstrous completely replicable functions.  If the conjecture is true, then there is some set of obstructions that keeps these series from being characters of elements in $G$.  For many of these functions, there are concrete positivity obstructions to their appearance as McKay-Thompson series, coming from orbifold conformal field theory considerations.  Most immediately, for any automorphism $g$ of a holomorphic vertex operator algebra with character $J$ (giving a McKay-Thompson series $T_g(\tau)$), there is a $g$-twisted module with character $T_g(-1/\tau)$ whose graded pieces must have non-negative dimension.  Similarly, there are more subtle positivity constraints from the Weil representation attached to the corresponding abelian intertwining algebra.  It is an interesting and natural question to consider how such positivity constraints interact with existing characterizations of the Monstrous Moonshine functions, such as in \cite{CMS04} and \cite{DF11}.  As far as I can tell, the positivity considerations coming from these other cusp expansions are invisible at the Lie algebra level.  More to the point, I am unaware of any substantial obstructions to the existence of non-monstrous q-compatible symmetries of $\fm$ that yield these series.

A second piece of evidence against the conjecture is that we have a concrete example of a Fricke Lie algebra $\fg$ with q-compatiblity groupoid coming from quantization, such that $\Aut^q(\fg)$ is strictly larger than the group of symmetries of the abelian intertwining algebra.  This means the ``add a torus and quantize'' process is not faithful on isomorphisms in general.  We note that a second, more severe example would come from a $C_2$-cofinite counterexample $V$ to the Frenkel-Lepowsky-Meurman uniqueness conjecture, because $V$ cannot admit any non-Fricke involutions, despite such q-compatible symmetries existing in the corresponding Lie algebra.  Naturally, one may instead view this as evidence that $V$ does not exist.


\section{Hauptmoduln in Generalized Monstrous Moonshine}

In 1987, Norton proposed an enhancement of the Monstrous Moonshine conjecture, on the basis of computations by Queen \cite{Q81} and later by himself.

\begin{conj} (Generalized Moonshine \cite{N87}, revised in \cite{N01})
There exists a rule that assigns to each element $g$ of the Monster simple group $\bM$ a graded projective representation $V(g) = \bigoplus_{n \in \bQ} V(g)_n$ of the centralizer $C_{\bM}(g)$, and to each pair $(g,h)$ of commuting elements of $\bM$ a holomorphic function $Z(g,h,\tau)$ on the complex upper half-plane $\fH$, satisfying the following conditions:
\begin{enumerate}
\item There is some lift $\tilde{h}$ of $h$ to a linear transformation on $V(g)$ such that
\[ Z(g,h,\tau) = \sum_{n \in \bQ} Tr(\tilde{h}|V(g)_n) q^{n-1}. \]
\item $Z(g,h,\tau)$ is invariant up to constant multiplication under simultaneous conjugation of the pair $(g,h)$ in $\bM$.
\item $Z(g,h,\tau)$ is either constant, or a Hauptmodul for some genus zero congruence group.
\item For any $\left( \begin{smallmatrix} a & b \\ c & d \end{smallmatrix} \right) \in SL_2(\bZ)$, $Z(g^a h^c, g^b h^d, \tau)$ is proportional to $Z(g,h, \frac{a \tau + b}{c \tau + d})$.
\item $Z(g,h,\tau) = J(\tau) = \frac{E_4(\tau)^3}{\Delta(\tau)} - 744 = q^{-1} + 196884 q + \cdots$ if and only if $g=h=1 \in \bM$.
\end{enumerate}
\end{conj}

Claims 1, 2, 4, and 5 were given a physical interpretation in \cite{DGH88}: $V(g)$ is the ``Hilbert space twisted by $g$'' in a conformal field theory with $\bM$ symmetry, and $Z(g,h)$ is a genus one partition function for a torus twisted in ``space'' by $g$ and in ``time'' by $h$.  Each of these physical ideas has a straightforward restatement in the language of vertex operator algebras, e.g., one may set $V(g)$ to be an irreducible $g$-twisted $V^\natural$-module.  Claims 1, 2, and 5 were proved in \cite{DLM00}, in part by showing that irreducible $g$-twisted $V^\natural$-modules exist and are unique up to isomorphism.  Furthermore, Claim 4 was reduced to a $g$-rationality hypothesis, which was finally resolved in \cite{CM16}.

When it comes to interpreting the Hauptmodul claim (3), physics does not give us much help yet.  For the original Monstrous Moonshine Conjecture, the genus zero condition would follow from conjectured uniqueness properties of $V^\natural$, as shown in \cite{T95}.  In fact, the necessary properties are now essentially known, thanks to the holomorphic orbifold work in \cite{vEMS}: Theorem 1 of \cite{PPV17} gives a proof that non-anomalous Fricke orbifolds of $V^\natural$ are isomorphic to $V^\natural$, and the claim that non-anomalous non-Fricke orbifolds yield $V_\Lambda$ appears to be a straightforward but laborous character calculation by the uniqueness results of \cite{DM04}.  On a more physical level of rigor, there is a very recent interpretation of the Hauptmodul condition in the Monstrous Moonshine conjecture in terms of T-duality on a heterotic string model in \cite{PPV16}, but it is not yet clear how this can be applied to Generalized Moonshine functions.  In lieu of a compelling physical interpretation in the general case, it is natural to look to Borcherds's proof of the original Monstrous Moonshine conjecture as a blueprint for a proof.  H\"ohn did precisely that when he showed that $Z(g,h,\tau)$ is a Hauptmodul whenever $g$ lies in class 2A \cite{H03}.  In proving this particular case of Generalized Monstrous Moonshine, H\"ohn also described precisely what needed to be proved to extend Borcherds's method to all relevant cases.  The outline is roughly the following:
\begin{enumerate}
\item Build an abelian intertwining algebra structure on the direct sum $\bigoplus_{i=0}^{|g|-1} V^\natural(g^i)$.
\item Add a spacetime torus: Make a conformal vertex algebra of central charge 26 by taking a graded tensor product with the abelian intertwining algebra of a suitable hyperbolic lattice.
\item Quantize: Apply a bosonic string quantization functor to the conformal vertex algebra.  This yields a Lie algebra $\fm_g$ equipped with a canonical projective action of $C_{\bM}(g)$ by automorphisms.
\item Generate a Lie algebra $L_g$ whose denominator identity is some automorphic infinite product.  This Lie algebra has a ``nice shape'', in the sense that its simple roots and homology are well-controlled.
\item Comparison: Show that the Lie algebra $\fm_g$ is isomorphic to $L_g$.  By \textit{transport de structure}, we obtain a Lie algebra with both a group action and nice shape.
\item Hauptmodul conclusion: Use the twisted denominator identity to produce recursion relations on the characters that are strong enough to conclude that the characters are Hauptmoduln.
\end{enumerate}
The first step was done in the landmark paper \cite{vEMS}, steps 2,3, and 5 were done in \cite{GM4}, step 4 was done in \cite{GM2}, and step 6 was done in \cite{GM1}.  We shall describe the last step in more detail.

We begin the last step with a $\bZ \times \frac{1}{N}\bZ$-graded Borcherds-Kac-Moody Lie algebra $L_g$ with an action of a central extension $\widetilde{C_{\bM}(g)}$ of $C_{\bM}(g)$ by automorphisms.  Here, and for the remainder of the paper, $N$ is the level of $T_g$ as a modular function, and it is an integer multiple of $|g|$ that divides $12|g|$.  The Weyl-Kac-Borcherds denominator formula of $L_g$ is given by the infinite product identity
\[ T_g(\sigma) - T_g(-1/\tau) = p^{-1}\prod_{m \in \bZ_{>0}, n \in \frac{1}{N} \bZ} (1-p^m q^n)^{c^g_{m,Nn}(mn)} \]
where the exponents $c^g_{m,Nn}(mn)$ are the coefficients of a vector-valued modular function describing the character of the abelian intertwining algebra, and $p = e^{2 \pi i \sigma}$.  This denominator formula is equivalent to the isomorphism $H_*(E) \cong \bigwedge^*(E)$ of virtual vector spaces, where $E = \bigoplus_{m \in \bZ_{>0}, n \in \frac{1}{N}\bZ} E_{m,n}$ is the positive subalgebra of $L_g$.  By promoting this to an isomorphism of virtual $\widetilde{C_{\bM}(g)}$-representations and taking characters, we obtain a twisted denominator identity
\[ \begin{aligned}
p^{-1} &+ \sum_{m > 0} \mathrm{Tr}(h|E_{1,-1/N}) \mathrm{Tr}(h|E_{m,1/N}) p^m - \sum_{n \in \frac{1}{N}\bZ} \mathrm{Tr}(h|E_{1,n}) q^n \\
&= p^{-1} \exp \left( - \sum_{i > 0} \sum_{m > 0, n \in \frac{1}{N}\bZ} \mathrm{Tr}(h^i | E_{m,n})p^{im}q^{in}/i \right)
\end{aligned} \]
for each $h \in \widetilde{C_{\bM}(g)}$.

At this point, we haven't specified any conditions the action of $\widetilde{C_{\bM}(g)}$ must satisfy, so this is by itself not enough to prove a genus zero property.  In the special case $g=1$, this is the question of specifying a good Monster action on the Monster Lie algebra $\fm$.  As it happens, the existence of an $\bM$-action on the Monster Lie algebra $\fm$ is not very special, because the Monster Lie algebra has many Monster actions, most of which are useless.  Indeed, by Lemma \ref{lem:diagram-automorphism-group}, there is a huge group of diagram automorphisms, given by applying arbitrary linear transformations to the simple root spaces.  Since the multiplicities of these root spaces are coefficients of the $J$-function, we have an action of the infinite dimensional group $GL_1(\bC) \times GL_{196884}(\bC) \times GL_{21493760}(\bC) \times \cdots$ on $\fm$, and there are many embeddings of the Monster.  The general case where $g$ ranges over Fricke elements of $\bM$ is quite similar.

For the Monster Lie algebra $\fm$, the good action comes from the quantization functor applied to the Monster vertex operator algebra $V^\natural$, whose automorphism group is $\bM$.  The recursions then arise from the identification of all homogeneous spaces whose bidegree has the same product.  For example, by identifying both $\fm_{2,2}$ and $\fm_{1,4}$ with $V_5$ as Monster modules, we obtain a decomposition $V_5 \cong V_4 \oplus \bigwedge^2 V_2$ of Monster representations, and a consequent identity on traces.  For the monstrous Lie algebras $L_g$, the identification is sparser, and matches root spaces with certain subspaces of twisted $V^\natural$-modules.

Once we have a suitably good action of $\widetilde{C_{\bM}(g)}$ on $L_g$, we obtain identities between the values of $\mathrm{Tr}(h^i | E_{m,n})$ for the various spaces $E_{m,n}$.  After some manipulation, we find that any equivariant Hecke operator $n\hat{T}_n$ sends $Z(g,h,\tau)$ to a monic polynomial of degree $n$ in $Z(g,h,\tau)$.  From there, a suitable adaptation of the methods in \cite{K94} and \cite{CG97} yield the Hauptmodul result.

In the next section, we will describe a new method to deduce the Hecke-monic property, where we dispense with the homology of the positive Lie subalgebra $E$, and consider instead the imaginary free subalgebra $\fu^+$.  The homology of free Lie algebras is easily computed, and doesn't need any structure theory of Borcherds-Kac-Moody Lie algebras.  

\section{Jurisich's decomposition} \label{sec:decomposition}

In the course of his proof of the Monstrous Moonshine conjecture, Borcherds employed a description of the homology of the positive subalgebra of the Monster Lie algebra $\fm$ that arose from an extension of the BGG-type results of \cite{GL76} to the Borcherds-Kac-Moody setting.  The details of this extension were not fully worked out until Jurisich's thesis a few years later, and they appear in published form as \cite{J04}.  The same extension was employed in \cite{GM2} to describe the homology of the Monstrous Lie algebras $L_g$.

In the case that $g$ is Fricke, one may avoid the complications of the BGG-type theory by using the fact that the Lie algebras have a large free component.  For the $g=1$ case with the Monster Lie algebra $\fm$, the mostly-free property is made precise in \cite{J98}.  Jurisich proves a decomposition $\fm = \fu^+ \oplus \fgl_2 \oplus \fu^-$, where $\fu^+$ and $\fu^-$ are freely generated by vector spaces $\cV^+$ and $\cV^-$.  Furthermore, $\cV^+$ is a $\fgl_2$-module generated by the imaginary simple roots of $\fm$.  The use of free Lie algebras makes computation simpler, because their homology is particularly easy to understand.  In particular, for $\fm$, the homology of the positive subalgebra is supported in degrees 0,1,2, but the homology of $\fu^+$ is supported in degrees 0 and 1, with $H_0(\fu^+,\bC) = \bC$ and $H_1(\fu^+,\bC) = \cV^+$.  This decomposition thus leads to simpler descriptions of twisted denominator formulas and removes the necessity of generalizing \cite{GL76}.  For the $g=1$ case, this method was used to derive complete replicability in \cite{JLW95}.

\begin{defn}
A generalized Cartan matrix is a matrix $A = (a_{i,j})_{i,j \in I}$ of real numbers, where
\begin{enumerate}
\item $I$ is a countable set.
\item $A$ is symmetrizable, i.e., there is a diagonal matrix $Q$ with positive real diagonal entries $Q_{i,i}$ such that $QA$ is symmetric.
\item $a_{i,j} < 0$ if $i \neq j$
\item If $a_{i,i} >0$, then $a_{i,i} = 2$ and $a_{i,j} \in \bZ$ for all $j \in I$.
\end{enumerate}
The universal Borcherds-Kac-Moody Lie algebra $\fg(A)$ of a generalized Cartan matrix is the Lie algebra generated by $\{h_i, e_i, f_i\}_{i \in I}$ subject to the relations
\begin{enumerate}
\item $\mathfrak{sl}_2$ relations: $[h_i, e_k] = a_{i,k} e_k$, $[h_i, f_k] = -a_{i,k} f_k$, $[e_i,f_j] = \delta_{i,j} h_i$.
\item Serre relations: If $a_{i,i} > 0$, then $\ad(e_i)^{1-2a_{i,j}}(e_j) = \ad(f_i)^{1-2a_{i,j}}(f_j) = 0$.
\item Orthogonality: If $a_{i,j} = 0$, then $[e_i, e_j] = [f_i,f_j] = 0$.
\end{enumerate}
A Borcherds-Kac-Moody algebra is a Lie algebra of the form $(\fg(A)/C).D$, where $C$ is a central ideal, and $D$ is a commutative Lie algebra of outer derivations.
\end{defn}

Borcherds-Kac-Moody algebras are a generalization of finite dimensional semisimple Lie algebras, where one relaxes some finiteness conditions and allows for simple roots with non-positive norm.  They are more or less characterized by the property that they are long and thin, i.e., all of the interesting cases admit a $\bZ$-grading with finite dimensional pieces.

\begin{defn}
A Lie algebra $\fg$ is Fricke if it
\begin{enumerate}
\item is Borcherds-Kac-Moody of rank 2 (i.e., the Cartan subalgebra is 2-dimensional),
\item has 1 real simple root,
\item has no norm zero simple roots,
\item admits a norm zero Weyl vector $\rho$ (meaning all simple roots $r$ satisfy $(r,\rho) = -r^2$), and
\item has only finitely many simple roots of norm greater than any fixed bound (i.e., for any $N>0$, only finitely many simple roots $r$ satisfying $(r,r) > -N$).
\end{enumerate}
\end{defn}

\begin{lem} \label{lem:basic-Fricke-properties} 
Let $\fg$ be a Fricke Lie algebra.  Then there is a decomposition as $\fg = \fu^+ \oplus \fgl_2 \oplus \fu^-$, where $\fu^+$ (resp. $\fu^-$) is freely generated by the underlying vector space of the $\fgl_2$-module $\cV^+$ (resp. $\cV^-$) that is generated (as a $\fgl_2$-module) by the imaginary simple roots of $\fg$ (resp. by the negatives of the imaginary simple roots).  Furthermore, $\fg$ admits a $\bZ \times \bZ$-grading, where the real simple root has degree $(1,-1)$, and the imaginary simple roots have degree $(1,n)$ for $n \geq 1$, and finite multiplicity.  
\end{lem}
\begin{proof}
The first claim, about the decomposition, is the rank 2 case of Corollary 5.1 in \cite{J98} (and does not use the Weyl vector hypothesis).

For the claim about the grading and simple roots, we note that if $\fg$ has no imaginary simple roots, we just have $\fg \cong \fgl_2$, and the conclusion is clear.  If $\fg$ has imaginary simple roots, then the root space is Lorentzian.  Without loss of generality, we may identify this space with $\bR^{1,1}$ with quadratic form $(a,b)^2 = -ab$, and place the Weyl vector at $(-1,0)$.  This forces all simple roots to lie on the line $\{(1,y)\}_{y \in \bR}$.  The real simple root $\alpha_0$ then lies at $(1,-1)$, since the nontrivial Weyl group element switches the $x$-axis and $y$-axis.  The imaginary simple roots then must lie at integer coordinates, because they generate finite dimensional $\fgl_2$ representations under the adjoint action, and finite-length $\alpha_0$-chains must start at integer coordinates.  The finite multiplicity follows from the last defining condition.
\end{proof}

Lemma 4.1 of \cite{GM2} asserts that for any power series of the form $f(q) = q^{-1} + O(q) \in q^{-1}\bZ[[q]]$ with non-negative integer coefficients, there is a Fricke Lie algebra whose denominator formula has the form
\[ f(p) - f(q) = p^{-1} \prod_{m \in \bZ_{>0}, n \in \bZ} (1-p^m q^n)^{c(m,n)}, \]
where $c(m,n)$ is the multiplicity of the root $(m,n) \in \bZ \times \bZ$, and the Lie algebra is unique up to isomorphism.  The following proposition is a converse, so we get a bijection between isomorphism classes of Fricke Lie algebras and power series of the above form.

\begin{prop}
Let $\fg$ be a Fricke Lie algebra, and let $f(q) \in \bZ((q)) = \sum_{n \in \bZ} c(1,n) q^n$, where $c(1,n)$ is the multiplicity of the simple root $(1,n)$.  Then the Weyl-Kac-Borcherds denominator formula for $\fg$ is
\[ \prod_{m \in \bZ_{>0}, n \in \bZ} (1-p^m q^n)^{c(m,n)} = p(f(p)-f(q)) \]
where $c(m,n)$ is the multiplicity of the root $(m,n) \in \bZ \times \bZ$.
\end{prop}
\begin{proof}
From Corollary 5.2 in \cite{J98}, the denominator formula is given by
\[ \prod_{\alpha \in \Delta_+} (1-e^\alpha)^{c(\alpha)} = \left(\sum_{w \in S_2} (-1)^{\epsilon(w)}e^{w\rho - \rho} \right) \left(1-\sum_{\alpha \in \Delta^{im}_+} r_\alpha e^\alpha \right), \]
where $r_\alpha$ is the multiplicity of $\alpha$ in $\cV^+$, $c(\alpha)$ is the multiplicity of $\alpha$ in $\fg$, and $\Delta^{im}_+$ is the set of positive imaginary roots.  By the $\bZ \times \bZ$-grading given in Lemma \ref{lem:basic-Fricke-properties}, we may set $e^\alpha = p^m q^n$ for $\alpha = (m,n)$ to get
\[ \prod_{m \in \bZ_{>0}, n \in \bZ} (1-p^m q^n)^{c(m,n)} = (1-pq^{-1})\left(1-\sum_{m>0,n>0} r_{m,n} p^m q^n\right). \]
We know that $r_{m,n} = r_{1,m+n-1}$ by virtue of the $\fgl_2$-structure of $\cV^+$.  The product on the right side therefore gives a telescoping sum that yields the formula we want.
\end{proof}

From these results, we see that Fricke Lie algebras are quite easy to describe recursively.  The Cartan matrix has the following block decomposition with constant blocks:
\[ \left( \begin{array}{ccccc}
2_{1 \times 1} & 0_{1 \times c(1,1)} & -1_{1 \times c(1,2)} & -2_{1 \times c(1,3)} & \cdots \\ 
0_{c(1,1) \times 1} & -2_{c(1,1) \times c(1,1)} & -3_{c(1,1) \times c(1,2)} & -4_{c(1,1) \times c(1,3)} & \cdots \\
-1_{c(1,2) \times 1} & -3_{c(1,2) \times c(1,1)} & -4_{c(1,2) \times c(1,2)} & -5_{c(1,2) \times c(1,3)} & \cdots \\
-2_{c(1,3) \times 1} & -4_{c(1,3) \times c(1,1)} & -5_{c(1,3) \times c(1,2)} & -6_{c(1,3) \times c(1,3)} & \cdots \\
\vdots & \vdots & \vdots & \vdots & \ddots
\end{array} \right). \]

As a $\bZ \times \bZ$-graded vector space, we have the following decomposition:
\[ \begin{array}{ccccccccc}
&  & & & & \vdots & \vdots & \vdots \\
& & & & & \bC^{c(1,3)} & \bC^{c(2,3)} & \bC^{c(3,3)} & \cdots \\
& & & & & \bC^{c(1,2)} & \bC^{c(2,2)} & \bC^{c(3,2)} & \cdots \\
& & & \bC & & \bC^{c(1,1)} & \bC^{c(2,1)} & \bC^{c(3,1)} & \cdots \\
& & & & \bC^2 & & & & \\
\cdots & \bC^{c(-3,-1)} & \bC^{c(-2,-1)} & \bC^{c(-1,-1)} & & \bC & & & \\
\cdots & \bC^{c(-3,-2)} & \bC^{c(-2,-2)} & \bC^{c(-1,-2)} & & & & & \\
\cdots & \bC^{c(-3,-3)} & \bC^{c(-2,-3)} & \bC^{c(-1,-3)} & & & & & \\
& \vdots & \vdots & \vdots & & & & 
\end{array} , \]
and the simple roots are spaces with $x$-coordinate 1, i.e., those in the column containing $\bC^{c(1,n)}$.  We also have the following decomposition of $\cV^+$ as a $\bZ_{>0} \times \bZ_{>0}$-graded vector space:
\[ \begin{array}{cccccc}
\vdots & \vdots & \vdots & \vdots & \vdots \\
\bC^{c(1,4)} & \bC^{c(1,5)} & \bC^{c(1,6)} & \bC^{c(1,7)} & \bC^{c(1,8)} & \cdots \\
\bC^{c(1,3)} & \bC^{c(1,4)} & \bC^{c(1,5)} & \bC^{c(1,6)} & \bC^{c(1,7)} & \cdots \\
\bC^{c(1,2)} & \bC^{c(1,3)} & \bC^{c(1,4)} & \bC^{c(1,5)} & \bC^{c(1,6)} & \cdots \\
\bC^{c(1,1)} & \bC^{c(1,2)} & \bC^{c(1,3)} & \bC^{c(1,4)} & \bC^{c(1,5)} & \cdots \\
\end{array} \]

\begin{prop} \label{prop:group-action-on-Lie algebra}
If $\fg$ is a Fricke Lie algebra, and $H$ is a group acting on $\fg$ by homogeneous automorphisms, then we have the following isomorphism of $H$-modules:
\[ \cV^+ \cong \fg_{1,-1} \oplus \bigoplus_{n =1}^\infty \bigoplus_{k=1}^n \fg_{1,n} \otimes \fg_{1,-1}^{\otimes k-1}. \]
In particular, the part of $\cV^+$ of degree $(m,n)$ is isomorphic to $\fg_{1,m+n-1} \otimes \fg_{1,-1}^{\otimes m-1}$.
\end{prop}
\begin{proof}
This follows straightforwardly from the adjoint $\fgl_2$-action on $\cV^+$.
\end{proof}

By following the argument of \cite{JLW95}, we can now form a twisted denominator formula.  Given a homogeneous action of a group $H$ of $\fg$, the Euler-Poincare identity $H(\fu^+,\bC) \cong \bigwedge \fu^+$ yields the equality $\log H(\fu^+,\bC) = \log \bigwedge \fu^+$ in $(R(H) \otimes \bQ)[[p,q]]$.  We may then expand the left side as $\log H(\fu^+,\bC) = \log(1 - \cV^+) = -\sum_{k=1}^\infty \frac{1}{k} (\cV^+)^k$ and the right side as $\log \bigwedge \fu^+ = -\sum_{k=1}^\infty \frac{1}{k} \psi^k(\fu^+)$, where $\psi^k$ is the $k$th Adams operation.

\begin{prop} \label{prop:first-twisted-denominator}
If $\fg$ is a Fricke Lie algebra, and $H$ is a group acting on $\fg$ by homogeneous automorphisms, then for any $h \in H$, we have the following trace relation:
\[ \sum_{k=1}^\infty \frac{1}{k} \left(\sum_{m,n=1}^\infty \Tr(h|\fg_{1,m+n-1}) \cdot \Tr(h|\fg_{1,-1})^{m-1} p^m q^n \right)^k = \sum_{m,n=1}^\infty \sum_{k=1}^\infty \frac{1}{k} \Tr(h^k | \fg_{m,n}) p^{km} q^{kn} \]
\end{prop}
\begin{proof}
This follows from applying the trace of $h$ to the identity
\[ -\sum_{k=1}^\infty \frac{1}{k} (\cV^+)^k = -\sum_{k=1}^\infty \frac{1}{k} \psi^k(\fu^+) \]
and applying the $H$-module decomposition in Proposition \ref{prop:group-action-on-Lie algebra}.
\end{proof}

\begin{cor} \label{cor:second-twisted-denominator}
Let $\fg$ be a Fricke Lie algebra, and let $H$ be a group acting on $\fg$ by homogeneous automorphisms.  For any $h \in H$, we let $\zeta_h = \Tr(h|\fg_{1,-1})$, and define $f_h(q) = \frac{1}{\zeta_h} \sum_{n \in \bZ} \Tr(h|\fg_{1,n}) q^n$.  Then we have the following trace relation:
\[ \log \frac{f_h(\zeta_h p)-f_h(q)}{(\zeta_h p)^{-1}-q^{-1}} = - \sum_{m,n=1}^\infty \sum_{k=1}^\infty \frac{1}{k} \Tr(h^k | \fg_{m,n}) p^{km} q^{kn} \]
\end{cor}
\begin{proof}
We expand the defining formula for Grunsky coefficients: 
\[ \begin{aligned}
\log \frac{f_h(p)-f_h(q)}{p^{-1}-q^{-1}} &= \log \frac{1 + \frac{1}{\zeta_h} \sum_{m = 1}^\infty \Tr(h|\fg_{1,m}) p^{m+1} - \frac{1}{\zeta_h} \sum_{n \in \bZ} \Tr(h|\fg_{1,n}) pq^n}{1-pq^{-1}} \\
&= \log \left( 1 - \sum_{n=1}^\infty \frac{\Tr(h|\fg_{1,n})}{\zeta_h} \sum_{r=0}^{n-1} p^{r+1} q^{n-r} \right) \\
&= \log \left( 1 - \sum_{m,n=1}^\infty \frac{\Tr(h|\fg_{1,m+n-1})}{\zeta_h} p^m q^n \right)
\end{aligned} \]
and change $p$ to $\zeta_h p$ to get $-1$ times the left side of the identity in Proposition \ref{prop:first-twisted-denominator}.
\end{proof}

This by itself does not express any information beyond the combination of Proposition \ref{prop:group-action-on-Lie algebra} and the fact that $\cV$ freely generates $\fu^+$.  As we mentioned earlier, in order to obtain the genus zero property for $f_h$, we need a constraint on the action of $H$, and this constraint will come from physics.

\section{No Ghost Theorem} \label{sec:no-ghost}

To produce a suitable group action on a Lie algebra, we apply an ``add a torus and quantize'' functor from a conformal-field-theoretic object that is already equipped with a group action.  For the case $g=1$, this is the Monster vertex operator algebra $V^\natural$, constructed in \cite{FLM88}.  For general elements $g$ with $T_g$ of level $N$, this is an abelian intertwining algebra ${}^g_N V^\natural \cong \bigoplus_{i=0}^{N-1} V^\natural(g^i)$ (in the sense of \cite{DL93}) formed from a sum of irreducible twisted $V^\natural$-modules, constructed in \cite{vEMS} with minor adjustments in \cite{GM4} to account for anomalous weights.  Our functor is the following composite:

\[ \xymatrix{ \text{Abelian intertwining algebra, $c=24$, $(\bZ \times \frac{1}{N}\bZ)/(N\bZ \times \bZ)$-graded} \ar[d]^{\text{add a torus}} \\
\text{Conformal vertex algebra $W$,  $c=26$, $\bZ \times \frac{1}{N} \bZ$-graded} \ar[d]^{\text{quantize}} \\
\text{Lie algebra $L_g$, $\bZ \times \frac{1}{N} \bZ$-graded} } \] 

The quantization functor is given by $W \mapsto P^{L_0 = 1}_W/rad(,)$, where $P^{L_0 = 1}_W$ denotes the space of Virasoro primary vectors of weight 1, i.e., vectors $v$ such that $L_i(v) = 0$ for $i>0$ and $L_0(v) = v$.  There is a naturally isomorphic functor, given by taking the BRST cohomology group $H^1_{BRST}(W)$.  Either functor satisfies the following properties (see \cite{GM4} section 3):
\begin{enumerate}
\item Inputs are Virasoro representations of central charge 26 equipped with a Virasoro-invariant bilinear form $(,)$, and outputs are vector spaces with nonsingular bilinear form.
\item Orthogonal direct sums are taken to orthogonal direct sums.
\item If the input is a conformal vertex algebra of central charge 26, then the output is a Lie algebra.  A vertex-algebra-invariant bilinear form is taken to a Lie-invariant bilinear form.
\item (No ghost theorem \cite{GT72}) Let $\pi^{1,1}_\lambda$ be the irreducible representation of the rank 2 Lorentzian Heisenberg algebra attached to $\lambda \in \bR^{1,1}$.  If the input has the form $V \otimes \pi^{1,1}_\lambda$ for $V$ a unitarizable Virasoro representation of central charge 24, and $\lambda \neq 0$, then the output is isomorphic to the subspace $V^{L_0 = 1-\lambda^2}$ of $V$ on which $L_0$ acts with eigenvalue $1-\lambda^2$.  If the input has the form $V \otimes \pi^{1,1}_0$, then the output is isomorphic to $(V^{L_0 = 0} \otimes_{\bR} \bR^{1,1}) \oplus V^{L_0 = 1}$.
\end{enumerate}
We briefly describe the nature of the isomorphism $P^{L_0 = 1}_{V \otimes \pi^{1,1}_\lambda}/rad(,) \to V^{L_0 = 1-\lambda^2}$ for $\lambda \neq 0$.  In the course of the proof, one constructs the following diagram of inclusions of subspaces:
\[ \xymatrix{ & (V \otimes \pi^{1,1}_\lambda)^{L_0 = 1} \\ P^{L_0 = 1} \ar[ur] & & K^{L_0 = 1} \ar[lu] \\ & T^{L_0 = 1} \ar[lu] \ar[ru] & & V^{L_0 = 1-\lambda^2} \otimes \bC |\lambda \rangle \ar[lu] }
\]
The bilinear form on $(V \otimes \pi^{1,1}_\lambda)^{L_0 = 1}$ is singular when restricted to $P^{L_0 = 1}$ and $K^{L_0 = 1}$, and nonsingular when restricted to $T^{L_0 = 1}$ and $V^{L_0 = 1-\lambda^2} \otimes \bC |\lambda \rangle$.  Furthermore, both $T^{L_0 = 1}$ and $V^{L_0 = 1-\lambda^2} \otimes \bC |\lambda \rangle$ are orthogonal complements to the radical of the form on $K^{L_0 = 1}$, while $T^{L_0 = 1}$ is an orthogonal complement to the radical of the form on $P^{L_0 = 1}$.  We then get canonical isomorphisms
\[ P^{L_0 = 1}_{V \otimes \pi^{1,1}_\lambda}/rad(,) \cong T^{L_0 = 1} \cong K^{L_0 = 1}/rad(,) \cong V^{L_0 = 1-\lambda^2} \otimes \bC |\lambda \rangle \]
The key point for us is that if we are given two vectors $\lambda$ and $\mu$ of the same norm in $\bR^{1,1}$, then a choice of isomorphism $\bC |\lambda \rangle \cong \bC |\mu \rangle$ induces a chain of isomorphisms
\[ P^{L_0 = 1}_{V \otimes \pi^{1,1}_\lambda}/rad(,) \cong V^{L_0 = 1-\lambda^2} \otimes \bC |\lambda \rangle \cong V^{L_0 = 1-\mu^2} \otimes \bC |\mu \rangle \cong P^{L_0 = 1}_{V \otimes \pi^{1,1}_\mu}/rad(,). \]
Naturally, we choose the unique isomorphism taking $|\lambda \rangle \mapsto |\mu \rangle$.

The ``add a torus'' functor is a way to construct a suitable input to the quantization functor from the abelian intertwining algebra ${}^g_N V^\natural$.  Each graded piece of ${}^g_N V^\natural$ has conformal weight given by the formula $(a,b/N) \mapsto ab/N$ modulo $\bZ$.  In order to form a conformal vertex algebra, we take a degree-wise tensor product (where we match degrees modulo $N\bZ \times \bZ$) with the abelian intertwining algebra $V_{I\!I_{1,1}(-1/N)}$ attached to the rational lattice $I\!I_{1,1}(-1/N)$, which can be identified with $\bZ \times \frac{1}{N}\bZ$ with quadratic form $(a,b/N)^2 = -ab/N$.  As a Heisenberg representation, we have the decomposition $V_{I\!I_{1,1}(-1/N)} \cong \bigoplus_{\lambda \in I\!I_{1,1}(-1/N)} \pi^{1,1}_\lambda$, and the summand of degree $\lambda$ has conformal weight $\lambda^2$.  The degree-wise tensor product cancels the non-integral conformal weights, and (after some adjustment with coboundaries) removes the braiding obstruction to locality.

Composing these two functors yields a $\bZ \times \frac{1}{N}\bZ$-graded Lie algebra $L_g$ whose degree $(0,0)$ space is identified with $\bR^{1,1} \otimes \bC$ and whose degree $(a,b/N)$ piece is identified with the subspace $({}^g_N V^\natural)_{a,b/N}^{L_0 = 1-ab/N}$ of ${}^g_N V^\natural$ made of eigenvectors of $L_0$ with eigenvalue $1-ab/N$, of degree $(a,b/N) + N\bZ \times \bZ$, when $(a,b/N) \neq (0,0)$.  We therefore have the following compatibility between root spaces:

\begin{prop}
Any abelian intertwining algebra automorphism of ${}^g_N V^\natural$ induces a corresponding homogeneous automorphism the Lie algebra $L_g$ that commutes with the identification of root spaces with subspaces of ${}^g_N V^\natural$.  In particular, if $H$ is a group acting on ${}^g_N V^\natural$ by abelian intertwining algebra automorphisms, then $H$ acts on the Lie algebra $L_g$ by homogeneous automorphisms, and the action on the degree $(a,b/N)$ piece depends only on the value of $ab/N$ and the residue class of $(a,b/N)$ modulo $N\bZ \times \bZ$.
\end{prop}

When $g$ is a Fricke element of $\bM$, i.e., the McKay-Thompson series $\sum_{n \in \bZ} \Tr(g|V^\natural_n) q^{n-1}$ is invariant under a Fricke involution $\tau \mapsto -1/N\tau$, then the resulting Lie algebra is Fricke.  It is therefore natural to define a notion of a symmetry of or an action on a Fricke Lie algebra that behaves like one that comes from an abelian intertwining algebra.  To generalize this phenomenon away from the concrete cases of Generalized Moonshine, we consider the following definitions:

\begin{defn}
Let $\fg$ be a Fricke Lie algebra graded by $\bZ \times \bZ$ as in Lemma \ref{lem:basic-Fricke-properties}.  A level $N$ q-compatibility groupoid on $\fg$ is a system of isomorphisms that identify all root spaces whose bidegree have both product and residue class modulo $N$ in common.  Equivalently, it is a collection of complex vector spaces $\{V^{a,b}_{n}\}_{a,b \in \bZ/N\bZ, n \in \bZ}$ together with isomorphisms $\{ \phi_{r,s}: \fg_{r,s} \simto V^{r,s}_{rs}\}_{(r,s) \neq (0,0)}$.  Given a Fricke Lie algebra equipped with a q-compatibility groupoid, a homogeneous automorphism of $\fg$ is q-compatible if it commutes with the given isomorphisms between root spaces.  We write $\Aut^q(\fg)$ to denote the group of all q-compatible automorphisms of $\fg$.
\end{defn}

We have a weaker, less rigidified version of this notion that still works well for studying characters.

\begin{defn}
We say that an action of a group $H$ on a Fricke Lie algebra (which we now grade by $\bZ \times \frac{1}{N}\bZ$) is level $N$ q-compatible if it is given by homogeneous automorphisms, and if the $H$-module structure on the degree $(a,b/N)$ piece depends only on the value of $ab/N$ and the residue classes of $a$ and $b$ modulo $N$.  
\end{defn}

If a group $H$ acts by symmetries of a Fricke Lie algebra equipped with choice of level $N$ q-compatibility groupoid, then the action is clearly level $N$ q-compatible.  We note that even when the group is trivial, q-compatibility of an action is an incredibly restrictive condition.  Indeed, Corollary \ref{cor:trivial-group-compatible-action} asserts that any infinite dimensional Fricke Lie algebra admitting a level $N$ q-compatible action of the trivial group has its simple roots described by a Hauptmodul with integer coefficients.

\section{Hecke operators} \label{sec:hecke}

We now combine the physical input with the twisted denominator formula.  If we have a Fricke Lie algebra $\fg$ with a level $N$ q-compatible action of a group $H$, we may replace all appearances of $\fg_{a,b}$ in the twisted denominator formula with some abstract $H$-modules $V^{a,b/N}_{ab/N}$, where the superscript is taken to lie in $(\bZ \times \frac{1}{N}\bZ)/(N\bZ \times \bZ)$.  

\begin{defn}
Let $\fg$ be a Fricke Lie algebra, and let $H$ be a group with a level $N$ q-compatible action on $\fg$.  For any $h \in H$, we define the normalized trace functions
\[f_{k,\ell,m}(q) = \sum_{n \in \frac{1}{N}\bZ} \sum_{\substack{r \in \frac{1}{N}\bZ/\bZ \\ kr - n \in \bZ}} e^{2\pi i \ell r} \Tr(h^m| V^{k,r}_n) q^n \]
for all triples of integers $(k,\ell,m)$.  In particular, $f_h(q) = \frac{1}{\zeta_h} f_{1,0,1}(q^N)$.
\end{defn}

\begin{thm} \label{thm:q-compatible-action-yields-hecke-operators}
Let $\fg$ be a Fricke Lie algebra, and let $H$ be a group with a level $N$ q-compatible action on $\fg$.  For any $h \in H$, we have
\begin{equation} \label{eq:q-compatible-action-yields-hecke-operators}
\log (p (f_{1,0,1}(\zeta_h p^N)-f_{1,0,1}(q))) = -\sum_{m=1}^\infty \frac{p^m}{m} \sum_{\substack{ad=m \\ 0 \leq b < d}} f_{d, -b,a}(e^{2\pi i b/d} q^{a/d})
\end{equation}
\end{thm}
\begin{proof}
Recall the equation
\begin{equation} \label{eq:second-twisted-denominator}
\log \frac{f_h(\zeta_h p)-f_h(q)}{(\zeta_h p)^{-1}-q^{-1}} = - \sum_{m,n=1}^\infty \sum_{k=1}^\infty \frac{1}{k} \Tr(h^k | \fg_{m,n}) p^{km} q^{kn}
\end{equation}
from Corollary \ref{cor:second-twisted-denominator}.  By adding $\log (1-\zeta_h p q^{-1})$ to both sides and changing $q$ to $q^{1/N}$, the left side of \eqref{eq:second-twisted-denominator} becomes equal to the left side of \eqref{eq:q-compatible-action-yields-hecke-operators}, and the right side of \eqref{eq:second-twisted-denominator} becomes:
\[ \begin{aligned}
 \log (1&-\zeta_h p q^{-1/N}) - \sum_{m=1}^\infty \sum_{n \in \frac{1}{N} \bZ_{>0}} \sum_{k=1}^\infty \frac{1}{k} \Tr(h^k | V^{m,n}_{mn}) p^{km} q^{kn} \\
&= -\sum_{m=1}^\infty p^m \sum_{ad=m} \frac{1}{a} \sum_{n \in \frac{1}{N} \bZ} \Tr(h^a | V^{d,n}_{dn}) q^{an}  \\
&= -\sum_{m=1}^\infty p^m \sum_{ad=m}\frac{1}{a} \sum_{n \in \frac{d}{N}\bZ} \sum_{\substack{r \in \frac{1}{N}\bZ/\bZ \\ n\in dr+d\bZ}} \Tr(h^a | V^{d,r}_n) q^{an/d}
\end{aligned}\]
where we use the q-compatible property to identify $\fg_{m,n}$ with $V^{m,n/N}_{mn/N}$.

Meanwhile, the right side of \eqref{eq:q-compatible-action-yields-hecke-operators} can be expanded as
\[ \begin{aligned}
-\sum_{m=1}^\infty & \frac{p^m}{m} \sum_{\substack{ad=m \\ 0 \leq b < d}} f_{d,-b,a}(e^{2\pi i b/d} q^{a/d}) \\
&= -\sum_{m=1}^\infty p^m \sum_{\substack{ad=m \\ 0 \leq b < d}} \frac{1}{a} \frac{1}{d} \sum_{n \in \frac{1}{N}\bZ} \sum_{\substack{r \in \frac{1}{N}\bZ/\bZ \\ n \in dr + \bZ}} e^{-2\pi i b r} \Tr(h^a| V^{d,r}_n) e^{2\pi i bn/d} q^{an/d} \\
\end{aligned}\]
The two expressions are equal because for any $r \in \frac{1}{N}\bZ/\bZ$ satisfying $n \in dr + \bZ$, we have
\[ \sum_{0 \leq b < d} e^{2\pi i b(\frac{n-dr}{d})} = \begin{cases} d & n-dr \in d\bZ \\ 0 & \text{otherwise} \end{cases}, \]
and this forces all summands indexed by $n \notin \frac{d}{N}\bZ$ to contribute zero.
\end{proof}

\begin{rem}
The end of this proof corrects the erroneous lines in the published versions of the proofs of \cite{GM1} Proposition 6.2 and \cite{GM2} Proposition 4.8.  In both cases, one has a line like
\[ - \sum_{m > 0} \sum_{ad=m} \frac{1}{a} \sum_{0 \leq b < d} \frac{1}{d} \sum_{n \in \frac{1}{N}\bZ} \underset{n \in dr+\bZ}{\sum_{r \in \frac{1}{N}\bZ/\bZ}} e(-br) \mathrm{Tr}(h^a | V^{d,r}_{1+n}) e(br) q^{an/d} p^m \]
that should be replaced with an argument of the form given here.
\end{rem}

Theorem \ref{thm:q-compatible-action-yields-hecke-operators} suggests the following definition is natural:

\begin{defn}
For each $n \in \bZ_{\geq 1}$ the equivariant Hecke operator $\hat{T}_n$ is defined by
\[ \hat{T}_n f_{k,\ell,m}(q) = \frac{1}{n} \sum_{\substack{ad=n\\0 \leq b < d}} f_{dk, a\ell-bk, am}(e^{2\pi i b/d} q^{a/d}). \]
We say that the collection $\{f_{k,\ell,m}(q)\}$ is Hecke-monic if $n\hat{T}_n f_{k,\ell,m}(q)$ is a monic polynomial in $f_{k,\ell,m}(q)$ of degree $n$ for all $n \geq 1$.
\end{defn}

These Hecke operators first appeared in the algebraic topology literature as cohomology operations on elliptic cohomology, e.g., in \cite{G09}.  They were independently discovered in \cite{T10}, where the Hecke-monic condition appeared in the guise of a "Generalized Moonshine replication formula".  We note that they obey essentially the same algebraic relations as ordinary Hecke operators (cf. Lemma 1.2 of \cite{GM1} and equation (17) of \cite{T10}).

\begin{prop} \label{prop:Hecke-monic}
Let $\fg$ be a Fricke Lie algebra, and let $H$ be a group with a level $N$ q-compatible action on $\fg$.  For any fixed choice of $h \in H$, the power series $f_{1,0,1}(q)$ is Hecke-monic.
\end{prop}
\begin{proof}
For each $m>0$, the $p^m$ term on the left side of Theorem \ref{thm:q-compatible-action-yields-hecke-operators} has coefficient $1/m$ times a monic polynomial of degree $m$ in $f_{1,0,1}(q)$.  The $p^m$ term on the right side is $\hat{T}_m f_{1,0,1}(q)$.  This implies $f_{1,0,1}(q)$ is Hecke-monic.
\end{proof}

We may now note a dictionary between our objects and those in \cite{GM1}.  The functions $f_{k,\ell,m}(q)$ correspond to the orbifold partition functions $Z(g^k,g^\ell h^m, \tau)$ where $g$ is the unique operator that takes any vector $v \in V^{k,r}_{kr}$ to $e^{2\pi i r}v$, and the space $V^{k,r}_{kr}$ here is written $V^{k,r}_{1+kr}$ there.  Invoking the main theorem of that paper, we find the following:

\begin{thm} \label{thm:q-compatible-action-gives-hauptmodul}
Let $\fg$ be a Fricke Lie algebra, and let $H$ be a group with a level $N$ q-compatible action on $\fg$.  For any finite order $h \in H$, the power series $f_{1,0,1}(q)$ is either a Hauptmodul or of the form $a q^{-1/N} + c q^{1/N}$, where $a$ is a root of unity and $c$ is either zero or a root of unity.
\end{thm}
\begin{proof}
The power series $f_{1,0,1}(q)$ is Hecke-monic with algebraic integer coefficients, so by Theorem 4.6 of \cite{GM1}, it either describes a Hauptmodul, or it has the form $aq^{-1} + b + cq$, where $a$ is a root of unity and $c$ is either a root of unity or zero.  Since $\fg$ has no norm zero simple roots, we have $b=0$.
\end{proof}

Functions of the form $aq^{-1} + b + cq$ are called ``trigonometric type'' in \cite{GM1}.  When $a=1$ and $b = 0$, these are sometimes called ``modular fictions''.

\begin{cor} \label{cor:trivial-group-compatible-action}
Let $\fg$ be a Fricke Lie algebra, and suppose the trivial group has a level $N$ q-compatible action on $\fg$.  If $\fg$ is finite dimensional, then it is isomorphic to either $\fgl_2$ or $\fsl_2 \times \fsl_2$.  If $\fg$ is infinite dimensional, then the simple root multiplicities of $\fg$ describe the $q$-expansion of a Hauptmodul.  Furthermore, this Hauptmodul uniquely determines the isomorphism type of $\fg$.
\end{cor}
\begin{proof}
The statements about the finite and infinite dimensional cases are immediate from the Theorem.  The uniqueness is from Lemma 4.1 of \cite{GM2}.
\end{proof}

\section{Groups of q-compatible symmetries}

In this section, we let $\fg$ be a Fricke Lie algebra equipped with level $N$ q-compatibility groupoid for some $N \in \bZ_{\geq 1}$.  As we mentioned in section \ref{sec:no-ghost}, these data naturally appear for Fricke Lie algebras that come from the ``add a torus and quantize'' functor.

\begin{lem} \label{lem:diagram-automorphism-group}
For each $n \in \bZ$, let $c(n)$ denote the multiplicity of the simple root of degree $(1,n)$ in $\fg$.  Then the group of homogeneous automorphisms is
\[ \prod_{n=-1}^\infty GL_{c(n)}(\bC)  \]
\end{lem}
\begin{proof}
Any homogeneous automorphism of $\fg$ is uniquely determined by its restriction to the simple roots.  Conversely, any homogeneous linear transformation on the sum of simple root spaces induces an automorphism of the free Lie algebra on $\{h_i, e_i, f_i\}_{i \in I}$ that is compatible with the relations defining $\fg$ - the only nontrivial step is checking Serre's relations.
\end{proof}

\begin{defn}
We write $\Aut^{q-inn}(\fg)$ for the group of q-compatible inner automorphisms of $\fg$.  Concretely, this is the subgroup of $GL_1(\bC) \times GL_1(\bC)$ whose elements are q-compatible, where the action of $GL_1(\bC) \times GL_1(\bC)$ is given by exponentiating the adjoint action of the Cartan subalgebra, i.e., $(\lambda,\mu) \in GL_1(\bC) \times GL_1(\bC)$ acts on $\fg_{r,s}$ as multiplication by $\lambda^r \mu^s$.
 \end{defn}

\begin{prop}
If $\fg$ is infinite dimensional, then $\Aut^{q-inn}(\fg)$ has finite order.
\end{prop}
\begin{proof}
Since $\fg$ is infinite dimensional, there exists a simple root $(1,n-1)$ for some $n > 2N$.  Then Proposition \ref{prop:group-action-on-Lie algebra} implies $\cV^+$ contains nonzero vectors $e_{a,n-a}$ in degrees $(a,n-a)$ for all $1 \leq a < n$.  Then the root $(nN, nN)$ has positive multiplicity because it contains the nonzero vector $e_{1,n-1} \wedge e_{2,n-2} \wedge \cdots \wedge e_{N,n-N} \wedge e_{n-N,N} \wedge \cdots \wedge e_{n-2,2} \wedge e_{n-1,1} \in \bigwedge^{2N} \cV^+$.  The level $N$ compatibility condition then yields isomorphic nonzero root spaces at $(N, n^2N)$ and $(n^2N,N)$, so for any $(\lambda,\mu) \in GL_1(\bC) \times GL_1(\bC)$, we have $\lambda^N \mu^{n^2N} = \lambda^{nN} \mu^{nN} = \lambda^{n^2N}\mu^N$, implying the three relations $\lambda^{(n-1)N} = \mu^{n(n-1)N}$ and $\lambda^{n(n-1)N} = \mu^{(n-1)N}$ and $\lambda^{(n^2-1)N} = \mu^{(n^2-1)N}$.  Substituting the $n$th power of the first relation into the second, we find that $\mu^{n^2(n-1)N} = \mu^{(n-1)N}$, so $\lambda^{(n^2-1)(n-1)N} = \mu^{(n^2-1)(n-1)N} = 1$.
\end{proof}

We can sharpen this substantially under rather mild conditions.

\begin{prop} \label{prop:inner-automorphisms}
If $\fg$ has a simple root at $(1,kN-1)$ for some $k \geq 3$, then any element of $\Aut^{q-inn}(\fg)$ has order dividing $N$.
\end{prop}
\begin{proof}
Proposition \ref{prop:group-action-on-Lie algebra} implies $\cV^+$ has positive dimension in degrees $(a,kN-a)$ for all $1 \leq a < kN$.  

In particular, the root spaces of degree $(N, kN-N)$ and $(kN-N,N)$ have nonzero multiplicity, and the level $N$ condition identifies them.  This implies any inner automorphism $(\lambda, \mu) \in GL_1(\bC) \times GL_1(\bC)$ satisfies $\lambda^N \mu^{kN-N} = \lambda^{kN-N}\mu^N$, so
\begin{equation} \label{eq:inner-first}
\lambda^{(k-2)N} = \mu^{(k-2)N}
\end{equation}

Because $\fu^+$ is freely generated by $\cV^+$, we may combine generators at $(1,kN-1)$ and $(3N-1,(k-3)N+1)$ to get a nonzero root space at $(3N,(2k-3)N)$.  The level $N$ condition identifies the root spaces at $(N, (6k-9)N)$, $(3N, (2k-3)N)$, $((2k-3)N,3N)$, and $((6k-9)N,N)$, so we obtain the equalities
\begin{equation} \label{eq:inner-second-stage-3N}
\lambda^N \mu^{(6k-9)N} = \lambda^{3N}\mu^{(2k-3)N} = \lambda^{(2k-3)N}\mu^{3N} = \lambda^{(6k-9)N}\mu^N.
\end{equation}
The first and third equalities yield $\mu^{(4k-6)N} = \lambda^{2N}$ and $\lambda^{(4k-6)N} = \mu^{2N}$, and combining these with the square of the second equality yields
\begin{equation} \label{eq:inner-second-stage-3N-conclusion}
\lambda^{8N} = \mu^{8N}.
\end{equation}

If $N>1$ or the root multiplicity at $(1,kN-1)$ is strictly greater than 1, then we may combine generators in $\cV^+$ at $(1,kN-1)$ and $(2N-1,(k-2)N+1)$ to get a nonzero root space at $(2N,(2k-2)N)$.  By a similar method we obtain the equalities
\begin{equation} \label{eq:inner-second-stage-2N}
\lambda^N \mu^{(4k-4)N} = \lambda^{2N}\mu^{(2k-2)N} = \lambda^{(2k-2)N}\mu^{2N} = \lambda^{(4k-4)N}\mu^N.
\end{equation}
and a similar deduction yields $\lambda^{3N} = \mu^{3N}$.  From \eqref{eq:inner-second-stage-3N-conclusion} we then get $\lambda^N = \mu^N$.  Then substitutions from \eqref{eq:inner-second-stage-2N} and \eqref{eq:inner-second-stage-3N} yield $\lambda^N  = \mu^N = 1$.

If $N=1$ and $(1,k-1)$ has multiplicity 1, then we first consider the case $k \geq 5$.  Then, $\cV^+$ has nonzero vectors $e_{a,k-a}$ in degree $(a,k-a)$ for $1 \leq a < k$, so the root space at $(k,k)$ contains the linearly independent vectors $e_{1,k-1} \wedge e_{k-1,1}$ and $e_{2,k-2} \wedge e_{k-2,2}$.  By level 1 compatibility, the multiplicity of the simple root $(1,k^2)$ is then at least 2.  We then substitute $k^2+1$ for $k$ and apply the previously considered case to deduce $\lambda = \mu = 1$.

If $N=1$, and $(1,k-1)$ has multiplicity 1, and $3 \leq k < 5$, then we still have a nonzero root space $(k,k)$, since it contains the vector $e_{1,k-1} \wedge e_{k-1,1}$.  Level 1 compatibility then yields a nonzero simple root $(1,k^2)$ and $k^2 > 5$.  Thus, we substitute $k^2+1$ for $k$ and apply the case $k \geq 5$ to deduce $\lambda = \mu = 1$.
\end{proof}

\begin{lem} \label{lem:q-compatible-automorphisms-give-genus-zero-functions}
For any finite order $g \in \Aut^q(\fg)$, the series $\sum_{n \in \bZ} \Tr(g|\fg_{1,n}) q^n$ is trigonometric type or a Hauptmodul.
\end{lem}
\begin{proof}
This follows immediately from Theorem \ref{thm:q-compatible-action-gives-hauptmodul}.
\end{proof}

We now focus on the case $\fg = \fm$ equipped with the level $1$ q-compatibility groupoid coming from $V^\natural$.  We note that by Proposition \ref{prop:inner-automorphisms}, the level 1 condition means there are no nontrivial q-compatible inner automorphisms of $\fm$.  As we mentioned in the introduction, $\Aut^q(\fm)$ contains $\bM$ by Borcherds's proof of Monstrous Moonshine, and we conjecture that $\Aut^q(\fm) \cong \bM$.  Evidence in favor of the conjecture amounts to results asserting that $\Aut^q(\fm)$ is ``not much larger'' than $\bM$.

\begin{lem} \label{lem:automorphisms-are-trivial-on-real-root}
Any element $g \in \Aut^q(\fm)$ acts by identity on the real simple root space $\fm_{1,-1}$.
\end{lem}
\begin{proof}
This follows immediately from the $\Aut^q(\fm)$-module isomorphism $\fm_{1,2} \cong \fm_{2,1}$ induced by the q-compatibility groupoid, together with the $\Aut^q(\fm)$-module isomorphism $\cV^+_{2,1} \cong \cV^+_{1,2} \otimes \fm_{1,-1}$ given in Proposition \ref{prop:group-action-on-Lie algebra}.
\end{proof}

\begin{prop} \label{prop:prime-order-automorphisms}
Let $g$ be an element of prime order $p$ in $\Aut^q(\fm)$.  Then $p$ divides $|\bM|$, i.e., $p \in \{2, 3, 5, 7, 11, 13, 17, 19, 23, 29, 31, 41, 47, 59, 71\}$.
\end{prop}
\begin{proof}
Let $T_g(q) = \sum_{n \in \bZ} \Tr(g|\fg_{1,n}) q^n$ be the character.  By Lemma \ref{lem:automorphisms-are-trivial-on-real-root}, this series has the form $q^{-1} + O(q)$, so from Section 9 of Borcherds's proof of the Monstrous Moonshine conjecture \cite{B92}, we know that $T_g(q)$ is completely replicable.  Then, by \cite{K94} and \cite{CG97}, it is either a modular fiction of the form $q^{-1} + aq$ or a Hauptmodul for a group $\Gamma_g \subset SL_2(\bR)$ satisfying the following properties:
\begin{enumerate}
\item $\Gamma_g$ contains $\Gamma_0(p^k)$ for some positive integer $k$.
\item $\fH/\Gamma_g$ is genus zero.
\item The intersection of $\Gamma_g$ with $\{ \left( \begin{smallmatrix} 1 & x \\ 0 & 1 \end{smallmatrix} \right) | x \in \bR \}$ is precisely the subset for which $x \in \bZ$.
\end{enumerate}
The replicates of $T_g(q)$ are given by characters of powers of $g$, and $T_1(q) = J(q)$.  This implies $T_g(q)$ cannot be a modular fiction, since replicates of modular fictions are all modular fictions.  It therefore suffices to classify $\Gamma_g$, and Ogg \cite{O74} showed that $p$ must then lie in the given list. 
\end{proof}

\begin{prop} \label{prop:injectivity-for-q-compatible-maps}
Let $c(n) = \dim \fm_{1,n}$ denote the simple root multiplicity.  The composite 
\[ \Aut^q(\fm) \hookrightarrow \prod_{n=-1}^\infty GL_{c(n)}(\bC) \to \prod_{n=1}^5 GL_{c(n)}(\bC) \]
is injective.  That is, any q-compatible automorphism of $\fm$ is uniquely determined by its restriction to the root spaces $\{ \fm_{1,n} \}_{n=1}^5$.
\end{prop}
\begin{proof}
We combine the isomorphism
\begin{equation} \label{eq:fricke-virtual-identity}
\sum_{k=1}^\infty \frac{1}{k} (\cV^+)^k = \sum_{k=1}^\infty \frac{1}{k} \psi^k(\fu^+)
\end{equation}
of graded virtual $\Aut^q(\fm)$-modules with the level 1 q-compatibility groupoid.  Let $n$ be a positive integer such that there exists a 4-tuple $(a,b,c,d)$ of positive integers satisfying the conditions $ab = cd$, $a+b = n+1$, and $c+d < n+1$.  A short calculation shows that the set of such positive integers $n$ is the complement of $\{1,2,3,5\}$.  Fixing such a choice of $(a,b,c,d)$, by equation \eqref{eq:fricke-virtual-identity}, $\fm_{a,b}$ is the direct sum of $\cV^+_{a,b}$ with a combination of tensor products of subspaces of $\cV^+$ whose bidegree sums to strictly less than $n+1$.  Equivalently, $\fm_{a,b}$ is isomorphic to the direct sum of the simple root space $\fm_{1,n}$ with a virtual $\Aut^q(\fm)$-module that is $\bQ$-linearly tensor generated by $\{ \fm_{1,m}\}_{1\leq m<n}$.  However, the level 1 q-compatibility groupoid contains an isomorphism from $\fm_{a,b}$ to $\fm_{c,d}$, and $\fm_{c,d}$ is a $\bQ$-linear combination of tensor products of subspaces $\cV^+$ whose bidegree sums to strictly less than $n+1$.  This implies the $\Aut^q(\fm)$-module structure of $\fm_{1,n}$ is completely determined by the $\Aut^q(\fm)$-module structure of the set $\{ \fm_{1,m}\}_{1\leq m<n}$.  Thus, the action of $\Aut^q(\fm)$ restricted to $\{ \fm_{1,n}\}_{n=1}^5$ completely determines the action on all simple root spaces, hence on $\fm$.
\end{proof}

\begin{rem}
We note that if we only worked with traces, hence completely replicable functions, we would only get this result for semisimple automorphisms.  It is only by applying the categorified version of complete replicability that we can show that unipotent automorphisms are identity.
\end{rem}

We now consider some cases where there are strictly more q-compatible automorphisms than expected from the vertex-algebraic input to the ``add a torus and quantize'' functor.  This behavior is an indication of how the functor loses information, and can be considered a weak form of evidence against the conjecture that $\Aut^q(\fm) \cong \bM$.

\begin{prop}
There exists a pair of abelian intertwining algebras ${}^g V^\natural$ and ${}^hV^\natural$ whose automorphism groups are not isomorphic, but yield isomorphic Fricke Lie algebras $L_g, L_h$ with equivalent q-compatibility groupoids under the ``add a torus and quantize'' functor.
\end{prop}
\begin{proof}
Let $g \in \bM$ lie in conjugacy class 27A, and let $h$ lie in 27B.  These two classes have identical McKay-Thompson series, hence their monstrous Lie algebras $L_g$ and $L_h$ have equal simple root multiplicities.  The canonical q-compatibility groupoids are then identified by any choice of isomorphisms of simple root spaces.  However, Table 2 of \cite{CN79} indicates the centralizers of these classes are not the same order, so the abelian intertwining algebras ${}^g V^\natural$ and ${}^hV^\natural$ have non-isomorphic automorphism groups.  Thus, we have two distinct natural q-compatible actions of groups on the same Lie algebra.
\end{proof}

We now briefly consider the Frenkel-Lepowsky-Meurman uniqueness conjecture.  The original conjecture asserts that $V^\natural$ is characterized by the following 3 properties:
\begin{enumerate}
\item $V^\natural$ is the only irreducible $V^\natural$-module.
\item $V^\natural$ has central charge 24.
\item $V^\natural_1 = 0$, i.e., $V^\natural$ has no nonzero vectors of weight 1.
\end{enumerate}
It is more convenient to consider a slightly weaker version of the conjecture, asserting that any holomorphic $C_2$-cofinite vertex operator algebra with character equal to $J$ is isomorphic to $V^\natural$.  Any counterexample to the original conjecture that is not a counterexample to this version would have to be quite exotic.

\begin{prop}
Let $V$ be a holomorphic $C_2$-cofinite vertex operator algebra whose character is $J$, such that $V \not\cong V^\natural$.  Then the following hold:
\begin{enumerate}
\item Applying the ``add a torus and quantize'' functor to $V$ yields a Fricke Lie algebra isomorphic to $\fm$ with equivalent level 1 q-compatibility groupoid.  In particular, we obtain an injection $\Aut(V) \to \Aut^q(\fm)$.
\item For any finite-order automorphism $g$ of $V$, the series $\Tr(g q^{L_0-1}|V)$ is the q-expansion of a completely replicable Hauptmodul.  In particular, we may classify finite-order elements of $\Aut(V)$ as Fricke or non-Fricke.
\item There is no non-Fricke involution on $V$.
\end{enumerate}
\end{prop}
\begin{proof}
The first result follows from Borcherds's proof of the Monstrous Moonshine conjectures \cite{B92}.  The complete replicability also follows from Borcherds's proof, and the Hauptmodul property follows from \cite{CG97}.

Now, let $g$ be a non-Fricke involution on $V$.  Non-Fricke completely replicable functions of prime order $p$ are classified in \cite{F93}, and are given by the Hauptmoduln for $\Gamma_0(p)$ for $p=2,3,5,7,13$.  In particular, $\Tr(g q^{L_0-1}|V) = \frac{\Delta(\tau)}{\Delta(2\tau)} + 24 = q^{-1} +276q - 2048q^2 + \cdots$.  Theorem 5.15 in \cite{vEMS} then gives us a $\bZ/2\bZ \times \bZ/2\bZ$-graded abelian intertwining algebra structure on $V \oplus V(g)$, such that the part indexed $\bZ/2\bZ \times \{0\}$ is a holomorphic $C_2$-cofinite vertex operator algebra $W$ equipped with an involution $\sigma$ giving the grading.  Furthermore, the $\sigma$-orbifold of $W$ is isomorphic to $V$.  A short computation shows that $W$ has character $J+24$, and by Theorem 3 of \cite{DM04}, $W$ is isomorphic to the Leech lattice vertex operator algebra $V_\Lambda$.  The involution $\sigma$ is the unique order 2 automorphism of $V_\Lambda$ with fixed point character $q^{-1} + 98580q + 10745856q^2 + \cdots$.  However the construction of $V^\natural$ in \cite{FLM88} is precisely by the $\sigma$-orbifold of $V_\Lambda$, so $V \cong V^\natural$.
\end{proof}

More generally, we expect a similar argument to imply the absence of all non-Fricke automorphisms of $V$ (and would be very interested to see a proof worked out in detail).  Even without this strengthening, we see that $V$ and $V^\natural$ have quite different automorphism groups despite admitting the same Fricke Lie algebra with q-compatibility groupoid.   Of course, it is perhaps more reasonable to view this discrepancy as evidence that $V$ does not exist.

\subsubsection*{Acknowledgments}

The author would like to thank Jim Lepowsky for suggesting the application of Jurisich's methods to Monstrous Lie Algebras, and Greg Moore for asking about the automorphism group of the Monster Lie algebra, both during the ``Mock modular forms, Moonshine, and string theory'' program at the Simons Center for Geometry and Physics in 2013.  This research was supported in part by the Program to Disseminate Tenure Tracking System, MEXT, Japan.

\end{document}